\documentclass[onethmnum,10pt]{siamltex}
\usepackage{booktabs}
\usepackage{amsfonts,amsmath,amssymb,epsfig,graphicx,algorithm,subfig}
\usepackage{hyperref}

\makeatletter
    \newcommand{\Rmn}[1]{\it{\expandafter\@slowromancap\romannumeral #1@}}
\makeatother

\newcommand{\R}{\mathbb{R}}
\newcommand{\Z}{\mathbb{Z}}
\newcommand{\B}{\mathcal{B}}

\newtheorem{innercustomthm}{Theorem}
\newenvironment{customthm}[1]
  {\renewcommand\theinnercustomthm{#1}\innercustomthm}
{\endinnercustomthm}

\title{Convergence of the Least Squares Shadowing Method for
Computing Derivative of Ergodic Averages}
\author{
Qiqi Wang \thanks{Department of Aeronautics and Astronautics,
MIT, 77 Mass Ave, Cambridge, MA 02139, USA}
}

\begin{document}
\maketitle
\begin{abstract}
For a parameterized hyperbolic system $u_{i+1} = f(u_i,s)$,
the derivative of an ergodic average
$\langle J\rangle = \underset{n\rightarrow\infty}{\lim} \frac1n \sum_1^n
J(u_i,s)$ to the parameter $s$ can be computed via the least squares
sensitivity method.  This method solves a constrained least squares
problem and computes an approximation to the desired derivative
$\frac{d\langle J\rangle}{ds}$ from the solution.
This paper proves that as the size of the least squares problem
approaches infinity, the computed approximation converges to the true
derivative.
\end{abstract}

\begin{keywords}
    Sensitivity analysis, linear response, least squares shadowing,
    hyperbolic attractor, chaos, statistical average, ergodicity
\end{keywords}

\begin{AMS}
\end{AMS}

\pagestyle{myheadings}
\thispagestyle{plain}
\markboth{Q. WANG}
{LEAST SQUARES SENSITIVITY}

\section{Introduction}
Consider a family of $C^1$ bijection maps $f(u,s):\R^m\times
\R\rightarrow\R^m$ parameterized by $s\in\R$.  We are also given a
$C^1$ function $J(u,s):\R^m\times\R\rightarrow\R$.
We assume that the system is \emph{ergodic}, i.e., the infinite time average
\begin{equation}
\langle J\rangle = \lim_{n\rightarrow\infty}
\frac1n \sum_{i=1}^n J(u_i,s)\;,
\quad\mbox{where}\quad u_{i+1} = f(u_i, s),\; i=1,\ldots
\end{equation}
depends on $s$ but does not depend on the initial state $u_0$.
The \emph{least squares shadowing} method attempts to compute its
derivative via
\begin{customthm}{LSS} \label{thm:lssmap}
Under ergodicity and hyperbolicity assumptions (details in Section
\ref{s:meanlss}),
\begin{equation} \label{lssmap}
\frac{d\langle J\rangle}{ds} = \lim_{n\rightarrow\infty}
\frac1n \sum_{i=1}^n (DJ(u_i,s))\, v_i^{\{n\}} + (\partial_sJ(u_i,s))\;,
\end{equation}
where $v_i^{\{n\}}\in\R^m, i=1,\ldots,n$ is the solution to the constrained
least squares problem
\begin{equation} \label{lss}
\min \frac12\sum_{i=1}^n v_i^{\{n\}T} v_i^{\{n\}}\quad\mbox{s.t.}\quad
v_{i+1}^{\{n\}} = (Df(u_i,s))\, v_i^{\{n\}} + (\partial_sf(u_i,s))\;,
\end{equation}
$i=1,\ldots,n-1$.
\end{customthm}
Here the linearized operators are defined as
\begin{equation}\begin{split}
(DJ(u,s))\, v &:= (D_vJ)(u,s) := \lim_{\epsilon\rightarrow 0}
\frac{J(u+\epsilon v,s) - J(u,s)}{\epsilon}\\
(Df(u,s))\, v &:= (D_vf)(u,s) := \lim_{\epsilon\rightarrow 0}
\frac{f(u+\epsilon v,s) - f(u,s)}{\epsilon}\\
(\partial_sJ(u,s)) &:= \lim_{\epsilon\rightarrow 0}
\frac{J(u,s+\epsilon) - J(u,s)}{\epsilon} \\
(\partial_sf(u,s)) &:= \lim_{\epsilon\rightarrow 0}
\frac{f(u,s+\epsilon) - f(u,s)}{\epsilon} \\
\end{split}\end{equation}
$(DJ), (\partial_sJ), (Df)$ and $(\partial_sf)$
are a $1\times m$ matrix, a scalar, an $m\times m$ matrix and an
$m\times 1$ matrix, respectively, representing the partial derivatives.

Computation of the derivative $d\langle J\rangle/ds$ represents a class
of important problems in computational science and engineering.
Many applications involve simulation of nonlinear
dynamical systems that exhibit chaos.  Examples include weather and
climate, turbulent combustion, nuclear reactor physics, plasma dynamics
in fusion, and multi-body problems in molecular dynamics.
The quantities that are to be predicted (the so-called quantities of
interest) are often time averages or expected values
$\langle J\rangle$.  Derivatives of these quantities of
interests to parameters are required in applications including
\begin{itemize}
\item {\bf Numerical optimization}.  The derivative of the objective
function $\langle J\rangle$ with respect to the design, parameterized by
$s$, is used by gradient-based algorithm to efficiently optimize in high
dimensional design spaces.
\item {\bf Uncertainty quantification}.  The derivative of the
quantities $\langle J\rangle$ with respect to the sources of
uncertainties $s$ can be used to assess the error and uncertainty
in the computed $\langle J\rangle$.
\end{itemize}

A scientific example is when the dynamical system is a climate model, and
the ergodic average $\langle J\rangle$ is the long time averaged
global mean temperature.  Its derivative to the amount of anthropogenic
emissions would be a valuable quantity to study.  An engineering example
can be found in simulation of turbulent air flow over an aircraft, where
the ergodic average $\langle J\rangle$ is the long time averaged drag.
Its derivative to shape parameters of the aircraft can help engineers
increase the efficiency of their design.  Although it is difficult to
analyze theoretically whether these complex dynamical systems are ergodic,
many of them have been observed to have ergodic quantities of interest,
leading to the popular \emph{chaotic
hypothesis}\cite{ruelle1980measures,gallavotti1995dynamical,evans2008statistical,gallavotti2006entropy}.
Efficient computation of the derivative of long time averaged quantities
in these systems is an important and challenging problem.

Traditional transient sensitivity analysis methods fail to compute
$d\langle J\rangle/ds$ in chaotic systems.  These methods focus on linearizing
initial value problems to obtain the derivative of the quantities of
interest.  When the quantity of interest is a
long-time average in a chaotic system, the derivative of this average
does not equal the long time average of the derivative.
As a result, traditional adjoint methods fail,
and the root of this failure is the
ill-conditioning of initial value problems of chaotic systems \cite{leaclimate}.

The differentiability of $\langle J\rangle$
has been shown by Ruelle \cite{springerlink10}.  Ruelle also constructed
a formula of the derivative.  However, Ruelle's formula is difficult to
compute numerically \cite{leaclimate, eyinkclimate}.  Abramov and Majda
are successful in computing the derivative based on the fluctuation
dissipation theorem \cite{0951-7715-20-12-004}.  However, for systems
whose SRB measure \cite{youngSRB} deviates strongly from Gaussian, 
fluctuation dissipation theorem based methods can be inaccurate.
Recent work by Cooper and Haynes has alleviated this limitation by
using a nonparametric method for estimating the stationary probability
density function \cite{cooper2011climate}.
Several more recent methods have been developed for computing this
derivative \cite{QJ:QJ200513160505, wangLorenz, bloniganwang, wanglcoadj}.
In particular,
the \emph{least squares shadowing} method \cite{wanglcoadj} is a
method that computes the derivative of $\langle J\rangle$
efficiently by solving a constrained least squares problem.
The primary advantage of this method is its simplicity.  The least
squares problem can be easily formulated and efficiently solved as
a linear system.  Compared to other methods, it is insensitive to the
dimension of the dynamical system and requires no knowledge of the
equilibrium probability distribution in the phase space.

This paper provides theoretical foundation for the least squares
sensitivity method by proving Theorem (\ref{thm:lssmap}) for uniformly
hyperbolic maps.  Section \ref{s:hyperbolic} lays out the basic assumptions,
and introduces hyperbolicity for readers who are not familiar with this
concept.  Section \ref{s:stability} then proves a special version of the
classic structural stability result, and defines the
shadowing direction, a key concept used in our proof.
Section \ref{s:meanshadow} demonstrates that the derivative of $\langle
J\rangle$ can be computed through the shadowing direction.
Section \ref{s:lssshadow} then shows that the least squares shadowing
method is an approximation of the shadowing direction.  We consider this
as a mathematically new and nontrivial result.
Section \ref{s:meanlss} finally proves Theorem \ref{thm:lssmap} by
showing that the approximation of the shadowing direction makes a
vanishing error in the computed derivative of $\langle J\rangle$.

\section{Uniform hyperbolicity}
\label{s:hyperbolic}

In this section we consider a dynamical system governed by
\begin{equation} \label{map}
u_{i+1} = f(u_i,s)
\end{equation}
with a parameter $s\in \R$, where $u_i\in \R^m$ and
$f:\R^m\times \R \rightarrow\R^m$ is $C^1$ and bijective in $u$.
This paper studies perturbation of $s$ around a nominal value.  Without
loss of generality, we assume the nominal value of $s$ to be 0.  We denote
$f^{(0)}(u,s) \equiv u$ and $f^{(i+1)}(u,s) \equiv f^{(i)}(f(u,s),s)$
for all $i\in \Z$.

We assume that the map has a
\emph{compact, global, uniformly hyperbolic attractor}
$\Lambda\subset\R^m$ at $s=0$, satisfying
\footnote{
A necessary condition for the applicability of our method is that the
dynamical system settles down to an attractor after many iterations.
The attractor can be a fixed point, a limit cycle, or a strange 
attractor.  Empirically, this means that the system eventually reaches an
equilibrium or quasi-equilibrium.
}
\begin{enumerate}
\item For all $u_0\in\R^m$,
$dist(\Lambda, f^{(n)}(u_0,0)) \xrightarrow{n\rightarrow\infty} 0$
where $dist$ is the Euclidean distance in $\R^m$.
\item There is a $C\in(0,\infty)$ and $\lambda\in(0,1)$, such that
for all $u\in\Lambda$, there is a splitting of $\R^m$
representing the space of perturbations around $u$.
\begin{equation} \label{hypermap}
\R^m = V^+(u) \oplus V^-(u)\;,
\end{equation}
where the subspaces are
\begin{itemize}
\item$V^+(u) :=
\{ v\in\R^m :
\|(Df^{(i)}(u,0))\, v\| \le C\,\lambda^{-i}\,\|v\|\;, \forall i<0 \}$
is the \emph{unstable subspace} at $u$, where $\|\cdot\|$ is the
Euclidean norm in $\R^m$, and
\[\begin{aligned}
(Df^{(i)}(u,s))\, v :=& \lim_{\epsilon\rightarrow 0}
\frac{f^{(i)}(u+\epsilon v,s) - f^{(i)}(u,s)}{\epsilon}\\
=& (Df^{(i-1)}(f(u,s),s))\; (Df(u,s))\; v
\end{aligned}\]
\item$V^-(u) :=
\{ v\in\R^m :
\|(Df^{(i)}(u,0))\, v\| \le C\,\lambda^{i}\,\|v\|\;, \forall i>0 \}$
is the \emph{stable subspace} at $u$.
\end{itemize}
Both $V^+(u)$ and $V^-(u)$ are continuous with respect to $u$.
\end{enumerate}

It can be shown that the subspaces $V^+(u)$ and $V^-(u)$ are
\emph{invariant} under the differential of the map $(Df)$, i.e., if
$u' = f(u,0)$ and  $v' = (Df(u,0))\, v$, then \cite{0951-7715-22-4-009}
\begin{equation} \label{invariant}
v\in V^+(u) \Longleftrightarrow v'\in V^+(u')\;,\quad
v\in V^-(u) \Longleftrightarrow v'\in V^-(u')\;.
\end{equation}

Uniformly hyperbolic chaotic dynamical systems are known as ``ideal chaos''.
Because of its relative simplicity, studies of hyperbolic chaos have
generated enormous insight into the properties of chaotic dynamical
systems \cite{kuznetsov2012hyperbolic}.
Although most dynamical systems encountered
in science and engineering are not uniformly hyperbolic, many of them
are classified as \emph{quasi-hyperbolic}.  These systems, including the
famous Lorenz system, have global properties
similar to those of uniformly hyperbolic systems \cite{bonatti2010dynamics}.
Results obtained on uniformly hyperbolic systems can often be
generalized to quasi-hyperbolic ones.  Scholars believe that very
complex dynamical systems like turbulence behave like they are
quasi-hyperbolic \cite{ruelle1980measures,gallavotti1995dynamical,evans2008statistical,gallavotti2006entropy}.
Although this paper focuses on proving the convergence of the least
squares shadowing method for uniformly hyperbolic systems, is has been
shown numerically that this method also works when the system is
not uniformly hyperbolic \cite{wanglcoadj}.

\section{Structural stability and the shadowing direction}
\label{s:stability}

The hyperbolic structure (\ref{hypermap}) ensures the
\emph{structurally stability}\cite{chaos}
of the attractor $\Lambda$ under perturbation in $s$.  Here we prove a
specialized version of the structural stability result.
\begin{theorem}\label{thm:map}
If (\ref{hypermap}) holds and $f$ is continuously differentiable,
then for all sequence $\{u_i^0,i\in\Z\}\subset\Lambda$ satisfying
$u_{i+1}^0=f(u_i^0,0)$, there is a $M>0$
such that for all $|s|<M$
there is a unique sequence $\{u_i^s,i\in\Z\}\subset\R^m$
satisfying $\|u_i^s-u_i^0\|<M$ and $u_{i+1}^s=f(u_i^s,s)$ for all
$i\in\Z$.  Furthermore, $u_i^s$ is $i$-uniformly
continuously differentiable to $s$.
\end{theorem}

Note: $i$-uniformly continuous differentiability of $u_i^s$
means $\forall s\in(-M,M)$ and $\epsilon>0: \exists\delta:
|s'-s|<\delta \Rightarrow
\left\|\frac{du^s_i}{ds}\big|_s - \frac{du^s_i}{ds}\big|_{s'}
\right\| < \epsilon$ for all $i$.  Other than the $i$-uniformly continuous
differentiability of $u_i^s$, this theorem can be obtained directly from
the shadowing lemma\cite{pilyugin1999shadowing}.
However, the uniformly continuous
differentiability result requires a more in-depth proof.  A more general
version of this result has been proven by Ruelle\cite{springerlink10}.

To prove the theorem, we denote $\mathbf{u} = \{u_i,i\in\Z\}$. The norm
\begin{equation}
\|\mathbf{u}\|_{\B} = \sup_{i\in\Z} \|u_i\|
\end{equation}
defines a Banach space $\B$ of
uniformly bounded sequences in $\R^m$.
Define the map $F:\B\times\R\rightarrow \B$
as $F(\mathbf{u},s) = \{u_i - f(u_{i-1},s),\;i\in\Z\}$.
We use the implicit function theorem to complete the proof,
which requires $F$ to be
differentiable and its derivative to be non-singular at $\mathbf{u}^0$.

\begin{lemma}
Under the conditions of Theorem \ref{thm:map},
$F$ has Fr\'echet derivative at all $\mathbf{u}\in\B$:
$$ (DF(\mathbf{u},s))\,\mathbf{v}
= \{v_i - (Df(u_{i-1},s))\, v_{i-1}\}\;,
\quad\mbox{where}\quad \mathbf{v} = \{v_i\}$$
\end{lemma}
\begin{proof}
Because $\|\mathbf{u}\|_{\B} = \sup_i \|u_i\|<\infty$, we can find
$C>2\|u_i\|$ for all $i$.  Because $f\in C^1$, its derivative
$(Df)$ is uniformly continuous in the compact set $\{u:\|u\|\le C\}$.
For $\|\mathbf{v}\|_{\B}<C/2$, we apply
the mean value theorem to obtain
$$\frac{f(u_i+ v_i,s) - f(u_i,s)}{\|\mathbf{v}\|_{\B}}
- \frac{(Df(u_i,s))\,v_i}{\|\mathbf{v}\|_{\B}} =
\frac{(Df(u_i+\xi v_i,s)) - (Df(u_i,s))}{\|\mathbf{v}\|_{\B}}\, v_i $$
where $0\le\xi\le 1$.  Because
$\|u_i+\xi v_i\|\le\|u_i\|+\|v_i\|<C$ for all $i$,
uniform continuity of $(Df)$ implies that $\forall\epsilon>0,
\exists\delta$ such that for all $\sup\|v_i\|<\delta$,
$$
\left\|\frac{(Df(u_i+\xi v_i,s)) - (Df(u_i,s))}{\|\mathbf{v}\|_{\B}}\,
v_i\right\| \le
\|Df(u_i+\xi v_i,s)) - (Df(u_i,s)\| < \epsilon$$
for all $i$.  Therefore,
\[ \begin{aligned}
\frac{F(\mathbf{u}+\mathbf{v},s) - F(\mathbf{u},s)}{\|\mathbf{v}\|_{\B}}
=& \left\{\frac{v_i}{\|\mathbf{v}\|_{\B}}
        - \frac{f(u_{i-1}+ v_{i-1},s) - f(u_{i-1},s)}{\|\mathbf{v}\|_{\B}}
\right\}\\
\longrightarrow\;&
\frac{\big\{v_i - (Df(u_{i-1},s))\,v_{i-1}\big\}}{\|\mathbf{v}\|_{\B}}
\end{aligned} \]
in the $\B$ norm.  Now we only need to show that the linear map
$\{v_i\}\rightarrow \{v_i - (Df(u_{i-1},s))\, v_{i-1}\}$ is bounded.
This is because
$(Df)$ is continuous, thus it is uniformly bounded in the compact set
$\{u:\|u\|\le C\}$.  Denote the bound in this compact set as $\|(Df)\| < A$,
then $ \big\|\{v_i - (Df(u_{i-1},s))\, v_{i-1}\}\big\|_{\B}
\le (1 + A)\; \|\{v_i\}\|_{\B}$.
\end{proof}

\begin{lemma}
Under conditions of Theorem \ref{thm:map},
the Fr\'echet derivative of $F$ at $\mathbf{u}^0$ and $s=0$ is a
bijection.
\end{lemma}
\begin{proof}
The Fr\'echet derivative of $F$ at $\mathbf{u}^0$ and $s=0$ is
$$ (DF(\mathbf{u}^0,0))\,\mathbf{v}
= \{v_i - (Df(u^0_{i-1},0))\, v_{i-1}\} $$
We only need to show that
for every $\mathbf{r}=\{r_i\}\in\B$, there exists a unique
$\mathbf{v}=\{v_i\}\in\B$ such that
$v_i - (Df(u^0_{i-1},0))\, v_{i-1} = r_i$ for all $i$.

Because of (\ref{hypermap}),
we can first split $r_i = r_i^+ + r_i^-$, where $r_i^+\in V^+(u^0_i)$ and
$r_i^-\in V^-(u^0_i)$.  Because $V^+(u)$ and $V^-(u)$ are continuous to $u$ and
$\Lambda$ is compact,
\[\inf_{\substack{u\in\Lambda\\r^{\pm}\in V^{\pm}(u)}}
\frac{\|r^++r^-\|}{\max(\|r^+\|,\|r^-\|)} = \beta > 0\;.
\]
(This is because if $\beta = 0$,
then by the continuity of $V^+(u),V^-(u)$ and the
compactness of $\big\{(u, r^{+}, r^{-}) \in \Lambda\times\R^m\times\R^m
: \max(\|r^+\|,\|r^-\|)=1\big\}$,
there must be a $u\in\Lambda, r^{+}\in V^{+}(u), r^{-}\in V^{-}(u)$ such
that $\max(\|r^+\|,\|r^-\|)=1$ and $r^++r^-=0$, which contradicts to the
hyperbolicity assumption (\ref{hypermap})).  Therefore,
$$\max(\|r_i^+\|,\|r_i^-\|) \le \frac{\|r_i\|}{\beta}
\le \frac{\|\mathbf{r}\|_{\B}}{\beta}\quad\mbox{ for all } i$$
Now let
$$ v_i = \sum_{j=0}^{\infty} (Df^{(j)}(u^0_{i-j},0))\, r_{i-j}^-
       - \sum_{j=1}^{\infty} (Df^{(-j)}(u^0_{i+j},0))\, r_{i+j}^+\;, $$
It can be verified
\footnote{
Combining
\[\sum_{j=0}^{\infty} (Df(u^0_{i-1}))(Df^{(j)}(u^0_{i-j-1}))\, r_{i-j-1}^-
= \sum_{j=0}^{\infty} (Df^{(j+1)}(u^0_{i-j-1}))\, r_{i-j-1}^-
= \sum_{j=1}^{\infty} (Df^{(j)}(u^0_{i-j}))\, r_{i-j}^- \]
and
\[\sum_{j=1}^{\infty} (Df(u^0_{i-1}))(Df^{(-j)}(u^0_{i+j-1}))\, r_{i+j-1}^+
\hspace{-1mm} =\hspace{-1mm} 
\sum_{j=1}^{\infty} (Df^{(-j+1)}(u^0_{i+j-1}))\, r_{i+j-1}^-
\hspace{-1mm} =\hspace{-1mm} 
\sum_{j=0}^{\infty} (Df^{(-j)}(u^0_{i+j}))\, r_{i+j}^-\]
we can obtain that
\[ v_i - (Df(u^0_{i-1}))\,v_{i-1}
 = v_i - \sum_{j=1}^{\infty} (Df^{(j)}(u^0_{i-j}))\, r_{i-j}^-
       + \sum_{j=0}^{\infty} (Df^{(-j)}(u^0_{i+j}))\, r_{i+j}^+\;,
 = r_{i}^- + r_{i}^+ = r_i\;.
\]
}
that $v_i - (Df(u^0_{i-1},0))\, v_{i-1} = r_i$,
and by the definition of $V^+(u)$ and $V^-(u)$,
\begin{equation}\label{biject}\begin{split}
\|v_i\| &\le \sum_{j=0}^{\infty} \left\|(Df^{(j)})(u^0_i)\, r_{i-j}^-\right\|
       + \sum_{j=1}^{\infty} \left\|(Df^{(-j)})(u^0_i)\, r_{i+j}^+\right\|\\
     &\le \sum_{j=0}^{\infty} C\,\lambda^{j} \|r_{i-j}^-\|
       + \sum_{j=1}^{\infty} C\,\lambda^{j} \|r_{i+j}^+\|
      \le\frac{2C}{1-\lambda} \frac{\|\mathbf{r}\|_{\B}}{\beta}\;,
\end{split}\end{equation}
Therefore, $v_i$ is uniformly bounded for all $i$.  Thus
$\mathbf{v}\in\B$.

Because of linearity, uniqueness of $\mathbf{v}$ such that 
$v_i - (Df(u^0_{i-1},0))\, v_{i-1} = r_i$ only need to be shown for
$\mathbf{r}=\mathbf{0}$.  To show this, we split $v_i = v^+_i + v^-_i$
where $v^+_i\in V^+(u^0_i)$ and $v^-_i\in V^-(u^0_i)$.  Because the spaces
$V^+(u^0_i)$ and $V^-(u^0_i)$ are invariant (Equation \ref{invariant}), 
\[ 0 = r_i = \left(v_i^+ - (Df(u^0_{i-1},0))\, v_{i-1}^+\right)
           + \left(v_i^- - (Df(u^0_{i-1},0))\, v_{i-1}^-\right) \]
where the two parentheses are in $V^+(u^0_i)$ and $V^-(u^0_i)$,
respectively.  Because $V^+(u^0_i)\cap V^-(u^0_i)=\{0\}$, both 
parentheses in the equation above must be 0 for all $i$, and
\[\left.
\begin{aligned}
v_i^+ &= (Df(u^0_{i-1},0))\, v_{i-1}^+ = \ldots =
(Df^{(i-j)}(u^0_j,0)\,v_j^+\\
v_i^- &= (Df(u^0_{i-1},0))\, v_{i-1}^- = \ldots =
(Df^{(i-j)}(u^0_j,0)\,v_j^-
\end{aligned}\right.\quad\mbox{for all } i>j\;.\]
By the definition of $V^+(u^0_i)$ and $V^-(u^0_i)$,
$\|v_j^+\| \le C\lambda^{i-j} \|v_i^+\|$,
$\|v_i^-\| \le C\lambda^{i-j} \|v_j^-\|$.
If $v_j^+\ne 0$ for some $j$, then
\[ \frac{\|v_i\|}{\beta} \ge \|v_i^+\|\ge \frac{\lambda^{j-i}}{C} \|v_j^+\|
\quad\mbox{for all}\quad i>j\;, \]
and $\{v_i,i\in\Z\}$ is unbounded.  Similarly, 
if $v_i^-\ne 0$ for some $i$, then
\[ \frac{\|v_j\|}{\beta} \ge \|v_j^-\|\ge \frac{\lambda^{j-i}}{C} \|v_i^-\|
\quad\mbox{for all}\quad j<i\;, \]
and $\{v_i,i\in\Z\}$ is unbounded.  Therefore, for $\{v_i\}$ to be
bounded, we must have $v_i = v_i^++v_i^-=0$ for all $i$.
This proves the uniqueness of $\mathbf{v}$ for $\mathbf{r} = \mathbf{0}$.
\end{proof}

\begin{proof}[Proof of Theorem \ref{thm:map}.]
$F(\mathbf{u}^0,0) = \{u^0_i - f(u^0_{i-1},0)\} = \mathbf{0}$.
So $\mathbf{u}^0$ is a zero point of $F$ at $s=0$.  The Combination of this
and the two lemmas enables application of the implicit function theorem.
Thus there exists $M>0$ such that for all $|s|<M$ there is a
unique $\mathbf{u}^s=\{u_i^s\}$ satisfying
$\|\mathbf{u}^s-\mathbf{u}^0\|_{\B} < M$ and
$F(\mathbf{u}^s,s)=0$.  Furthermore, $\mathbf{u}^s$ is continuously
differentiable to $s$, i.e., $\frac{d\mathbf{u}^s}{ds} \in \B$ is
continuous with respect to $s$ in the $\B$ norm.  By the definition of
derivatives (in $\B$ and in $\R^m$), $\frac{d\mathbf{u}^s}{ds} = \left\{
\frac{d u^s_i}{ds}\right\}$.  Continuity of $\frac{d\mathbf{u}^s}{ds}$ in
$\B$ then implies that $\frac{d u^s_i}{ds}$ is $i$-uniformly continuous with
respect to $s$.
\end{proof}

Theorem \ref{thm:map} states that for a series $\{u_i^0\}$ satisfying
the governing equation (\ref{map}) at $s=0$, there is a series
$\{u_i^s\}$ satisfying the governing equation at nearby values of $s$.
In addition, $u_i^s$ \emph{shadows} $u_i^0$, i.e., $u_i^s$ is close to
$u_i^0$ when $s$ is close to 0.
Also, $\left\{\frac{du_i^s}{ds}\big|_{s=0}\right\}$
exists and is $i$-uniformly bounded.

\begin{definition} \label{def:shadow}
The {\bf shadowing direction} $v_i^{\{\infty\}}$ is defined as the
uniformly bounded series
\[ \mathbf{v}^{\{\infty\}} := \left\{v_i^{\{\infty\}}\right\}
:= \left\{\frac{du_i^s}{ds}\Big|_{s=0}\right\}
= \frac{d\mathbf{u}^s}{ds}\Big|_{s=0} \in\B\;,\]
where $u_i^s$ is defined by Theorem \ref{thm:map}.
\end{definition}

The shadowing direction is the direction in which the shadowing series $u_i^s$
moves as $s$ increases from 0.  It provides a
vehicle by which we prove Theorem \ref{thm:lssmap}.  We show that
the derivative of the ergodic mean $\langle J\rangle$ to $s$ can be
obtained if the shadowing direction $v_i^{\{\infty\}}$ was given
(Section \ref{s:meanshadow}).
We then show that $v_i^{\{n\}}$, the solution to the constrained
least squares problem (\ref{lss}), sufficiently approximates the shadowing
direction $v_i^{\{\infty\}}$ when $n$ is large (Section \ref{s:lssshadow}).
We finally show in Section \ref{s:meanlss}) that the same derivative can
be obtained from the least squares solution $v_i^{\{n\}}$.

\section{Ergodic mean derivative via the shadowing direction}
\label{s:meanshadow}

This section proves an easier
version of Theorem \ref{thm:lssmap} that replaces the solution to the
constrained least squares problem
$v_i^{\{n\}},i=1,\ldots,n$ by the shadowing direction
$v_i^{\{\infty\}} = \frac{du_i^s}{ds}\big|_{s=0}$.

\begin{theorem} \label{thm:shadowmean}
If (\ref{hypermap}) holds and $f$ is continuously differentiable,
For all continuously differentiable function
$J(u,s):\R^m\times \R\rightarrow\R$ whose infinite time average
\begin{equation} \label{mapJ}
\langle J\rangle := \lim_{n\rightarrow\infty}\frac1n
\sum_{i=1}^n J(f^{(i)}(u_0,s),s)
\end{equation}
is independent of the initial state $u_0\in\R^m$,
let $\{v_i^{\{\infty\}}, i\in\Z\}$ be the sequence of shadowing
direction in Definition \ref{def:shadow}, then
\begin{equation} \label{mapJshadow}
\frac{d\langle J\rangle}{ds}\bigg|_{s=0} = \lim_{n\rightarrow\infty}
\frac1n \sum_{i=1}^n \left((DJ(u_i^0,0)) v_i^{\{\infty\}}
+ (\partial_sJ(u_i^0,0))\right)\;,
\end{equation}
\end{theorem}
\begin{proof}
This proof is essentially an exchange of limits through uniform
convergence.
Because $\langle J\rangle$ in Equation (\ref{mapJ}) independent of
$u_0$, we set $u_0=u^s_0$ in Theorem \ref{thm:map}
(thus $f^{(i)}(u_0^s,s)=u^s_i$) and obtain
\[
\frac{d\langle J\rangle}{ds}\bigg|_{s=0}
= \lim_{s\rightarrow 0}
\frac{\langle J\rangle|_{s=s} - \langle J\rangle|_{s=0}}{s}
= \lim_{s\rightarrow 0}\lim_{n\rightarrow\infty}\frac1n
\sum_{i=1}^n \frac{J(u_i^{s},s) - J(u_i^0,0)}{s}
\]
Denote
\[ \gamma^s_i = \dfrac{dJ(u_i^{s},s)}{ds}
 = (DJ(u_i^s,s)) \dfrac{du_i^s}{ds} + (\partial_sJ(u_i^s,s))\]
and use the mean value theorem, we obtain
\[
\frac{d\langle J\rangle}{ds}\bigg|_{s=0}
= \lim_{s\rightarrow 0}\lim_{n\rightarrow\infty}\frac1n
\sum_{i=1}^n \gamma^{\xi_i(s)}_i\;,
\;\mbox{ where all } |\xi_i(s)|\le|s|.
\]
Because $J$ is continuously differentiable, we can choose a compact
neighborhood of $\Lambda\times\{0\}\subset\R^m\times\R$ in which
both $(DJ(u,s))$ and $(\partial_sJ(u,s))$ are uniformly continuous.
When $s$ is sufficiently small, this neighborhood of $\Lambda\times\{0\}$
contains $(u_i^{s}, s)$
for all $i$ because $u_i^{0}\in\Lambda$ and $u_i^{s}$ are $i$-uniformly
continuously differentiable (from Theorem \ref{thm:map})
and therefore are $i$-uniformly continuous.  Also, $\dfrac{du_i^s}{ds}$
are $i$-uniformly continuous.  Therefore,
for all $\epsilon>0$, there exists $M>0$, such that for all
$|\xi|<M$,
$$\|\gamma^{\xi}_i - \gamma^0_i\|<\epsilon\quad\forall i.$$
Therefore, for all $|s|<M$, $|\xi_i(s)|\le|s|\le M$ for all
$i$, thus for all $n>0$,
\[ \left\|\frac1n \sum_{i=1}^n \gamma^{\xi_i(s)}_i
 - \frac1n \sum_{i=1}^n \gamma^0_i \right\|
 \le \frac1n \sum_{i=1}^n \left\|\gamma^{\xi_i(s)}_i - \gamma^0_i\right\|
 < \epsilon\;.
\]
thus,
\[ \left\|\lim_{n\rightarrow\infty}\frac1n \sum_{i=1}^n
\gamma^{\xi_i(s)}_i
 - \lim_{n\rightarrow\infty}\frac1n \sum_{i=1}^n \gamma^0_i \right\|
 \le \epsilon\;.
\]
Therefore,
\[ \frac{d\langle J\rangle}{ds}\bigg|_{s=0} = 
\lim_{s\rightarrow 0}\lim_{n\rightarrow\infty}\frac1n
\sum_{i=1}^n \gamma^{\xi_i(s)}_i
   = \lim_{n\rightarrow\infty}\frac1n \sum_{i=1}^n
   \gamma^0_i\;.
\]
This competes the proof via the definition of $\gamma_i^0$ and
$v_i^{\{\infty\}}$.
\end{proof}

With Theorem \ref{thm:shadowmean}, we are one step away from the main
theorem (Theorem \ref{thm:lssmap}) -- the shadowing direction
$v_i^{\{\infty\}}$ in Theorem \ref{thm:shadowmean} needs to
be replaced by the solution $v_i^{\{n\}}$ to the least squares problems
(\ref{lss}).  The next section proves a bound of the distance between
$v_i^{\{\infty\}}$ and $v_i^{\{n\}}$.

\section{Computational approximation of shadowing direction}
\label{s:lssshadow}

This section assumes all conditions of Theorem \ref{thm:map},
and focus on when $s=0$.  We denote $u_i^0$ by $u_i$
in this section and the next section.

The main task of this section is providing a bound for
\begin{equation} \label{shadowerr}
e_i^{\{n\}} = v_i^{\{n\}} - v_i^{\{\infty\}} \;,\quad i=1,\ldots,n
\end{equation}
where $v_i^{\{n\}}$ is the solution to the least squares problem
\begin{equation} \label{lss2}
\min \frac12\sum_{i=1}^n v_i^{\{n\}T} v_i^{\{n\}}\quad\mbox{s.t.}\quad
v_{i+1}^{\{n\}} = (Df(u_i,0))\, v_i^{\{n\}} + (\partial_sf(u_i,0)),
\quad i=1,\ldots,n-1.
\end{equation}
This is a mathematically new result in the following sense.  The
shadowing lemma guarantees the existence of a shadowing trajectory, but
provides no clear way to numerically compute it or its derivative.  This section
suggests that the solution to the least squares problem (\ref{lss2}) is a
useful approximation to the derivative of the shadowing trajectory, and
proves a bound of the approximation error.
This bound will then enable us to show that the difference
between $v_i^{\{n\}}$ and $v_i^{\{\infty\}}$ makes a vanishing difference in
Equation (\ref{mapJshadow}) as $n\rightarrow\infty$.

\begin{lemma} \label{thm:6}
$e_i^{\{n\}}$ as defined in Equation (\ref{shadowerr}) satisfy
\begin{equation} \label{erroreqn}
e_{i+1}^{\{n\}} = (Df(u_i,0))\, e_i^{\{n\}}\;,\quad i=1,\ldots,n-1
\end{equation}
In addition, their components in the stable and unstable directions,
$e_i^{\{n\}+}\in V^+(u_i)$ and $e_i^{\{n\}-}\in V^-(u_i)$, where
$e_i^{\{n\}+} + e_i^{\{n\}-} = e_i^{\{n\}}$, satisfies
\begin{equation} \label{erroreqn2}
e_{i+1}^{\{n\}+} = (Df(u_i,0))\, e_i^{\{n\}+}\;,\quad
e_{i+1}^{\{n\}-} = (Df(u_i,0))\, e_i^{\{n\}-}\;,\quad i=1,\ldots,n-1
\end{equation}
\end{lemma}
\begin{proof}
By definition, $u_{i+1}^s=f(u_i^s,s)$ for all $s$ in a neighborhood of
0.  By taking derivative to $s$ on both sides, we obtain
\[ v_{i+1}^{\{\infty\}} = (Df(u_i,0)) v_i^{\{\infty\}} +
(\partial_sf(u_i,0)) \]
Subtracting this from the constraint in Equation (\ref{lss2}), we obtain
Equation (\ref{erroreqn}).

By substituting $e_i^{\{n\}} = e_i^{\{n\}+} + e_i^{\{n\}-}$ into Equation
(\ref{erroreqn}), we obtain
\[ \left(e_{i+1}^{\{n\}+} - (Df(u_i,0))\, e_i^{\{n\}+}\right)
 + \left(e_{i+1}^{\{n\}-} - (Df(u_i,0))\, e_i^{\{n\}-}\right)
 = 0 \]
Because the spaces $V^+(u)$ and $V^-(u)$ are invariant
(Equation (\ref{invariant})),
\[(Df(u_i,0))\, e_i^{\{n\}\pm}\in V^{\pm}(u_{i+1}),\quad\mbox{thus}\quad
 \left(e_{i+1}^{\{n\}\pm} - (Df(u_i,0))\,
e_i^{\{n\}\pm}\right)\in V^{\pm}(u_{i+1})\;.
\]
Because they sum to 0, both parentheses must be in
$V^+(u_{i+1}) \cap V^-(u_{i+1}) = \{0\}$.  This proves Equation
(\ref{erroreqn2}).
\end{proof}

Lemma \ref{thm:6} indicates that for all $\epsilon^+$ and $\epsilon^-$,
\begin{equation} \label{vprime}
v_i'^{\{n\}} = v_i^{\{n\}} + \epsilon^+ e_i^{\{n\}+} + \epsilon^- e_i^{\{n\}-}
\end{equation}
satisfies the constraint in Problem (\ref{lss2}), i.e.,
\[ v_{i+1}'^{\{n\}} = (Df(u_i,0))\, v_i'^{\{n\}} + (\partial_sf(u_i,0)),
\quad i=1,\ldots,n-1\;.  \]
Because $v_i^{\{n\}}$ is the solution to Problem (\ref{lss2}),
it must be true that
\[ \sum_{i=1}^n v_i^{\{n\}T} v_i^{\{n\}} \le \sum_{i=1}^n v_i'^{\{n\}T}
v_i'^{\{n\}} \quad\mbox{ for all }\epsilon^+ \mbox{ and }\epsilon^-\;.
\]
By substituting the definition of $v'_i$ in Equation (\ref{vprime}),
and use the first order optimality condition with respect to
$\epsilon^+$ and $\epsilon^-$ at $\epsilon^+=\epsilon^-=0$, we obtain
\begin{equation} \label{opteps}
\sum_{i=1}^n v_i^{\{n\}T}e_i^{\{n\}+}
= \sum_{i=1}^n v_i^{\{n\}T} e_i^{\{n\}-} = 0
\end{equation}
By substituting
$ v_i^{\{n\}} = v_i^{\{\infty\}} + e_i^{\{n\}}
= v_i^{\{\infty\}} + e_i^{\{n\}+} + e_i^{\{n\}-} $
into Equation (\ref{opteps}), we obtain
\begin{equation}\label{equality} \begin{split}
&
\sum_{i=1}^n (v_i^{\{\infty\}})^T e_i^{\{n\}+} +
\sum_{i=1}^n (e_i^{\{n\}+})^T e_i^{\{n\}+} +
\sum_{i=1}^n (e_i^{\{n\}-})^T e_i^{\{n\}+} = 0 \\
&
\sum_{i=1}^n (v_i^{\{\infty\}})^T e_i^{\{n\}-} +
\sum_{i=1}^n (e_i^{\{n\}+})^T e_i^{\{n\}-} +
\sum_{i=1}^n (e_i^{\{n\}-})^T e_i^{\{n\}-} = 0 \\
\end{split} \end{equation}
To transform Equation (\ref{equality}) into bounds on $e_i^{\{n\}+}$ and
$e_i^{\{n\}-}$, we need the following lemma.

\begin{lemma}\label{thm:ineq}
The hyperbolic splitting of $e_i^{\{n\}}$ as defined in Equation (\ref{shadowerr})
satisfies
\[ \|e_i^{\{n\}+}\| \le C\,\lambda^{n-i} \|e_n^{\{n\}+}\| \;,\quad
   \|e_i^{\{n\}-}\| \le C\,\lambda^{i}  \|e_0^{\{n\}-}\| \]
\end{lemma}
\begin{proof}
This is a direct consequence of Equation (\ref{erroreqn2}) and the
definition of $V^+$ and $V^-$ in Equation (\ref{hypermap}).
\end{proof}

By combining the first equality in Equation (\ref{equality})
with Lemma \ref{thm:ineq} and using the Cauchy-Schwarz inequality, we obtain
\[\begin{aligned}
 \|e_n^{\{n\}+}\|^2 &\le \sum_{i=1}^n (e_i^{\{n\}+})^T e_i^{\{n\}+}
 = -\sum_{i=1}^n (v_i^{\{\infty\}})^T e_i^{\{n\}+}
   -\sum_{i=1}^n (e_i^{\{n\}-})^T e_i^{\{n\}+} \\
&\le \sum_{i=1}^n \|v_i^{\{\infty\}}\| \|e_i^{\{n\}+}\|
 +  \sum_{i=1}^n \|e_i^{\{n\}-}\|\|e_i^{\{n\}+}\| \\
&\le \sum_{i=1}^n  C\,\lambda^{n-i}\|v_i^{\{\infty\}}\| \|e_n^{\{n\}+}\|
 +  \sum_{i=1}^n C^2\lambda^{n}\|e_0^{\{n\}-}\|\|e_n^{\{n\}+}\|
 \end{aligned} \]
Therefore,
\[ \|e_n^{\{n\}+}\| \le \frac{C}{1-\lambda}
\left\|\mathbf{v}^{\{\infty\}}\right\|_{\B} 
+ nC^2\lambda^n \|e_0^{\{n\}-}\| \]
where the $\B$ norm is as defined in Section \ref{s:stability}, and
is finite by Theorem \ref{thm:map}.
Similarly, by combining the second equality in Equation (\ref{equality})
with Lemma \ref{thm:ineq} and using the Cauchy-Schwarz inequality, 
\[ \|e_0^{\{n\}-}\| \le \frac{C}{1-\lambda} 
\left\|\mathbf{v}^{\{\infty\}}\right\|_{\B}
+ nC^2\lambda^n \|e_n^{\{n\}+}\| \]
When $n$ is sufficiently large such that $nC^2\lambda^n<\frac13$, we can
substitute both inequalities into each other and obtain
\begin{equation}\label{ineq}
\|e_n^{\{n\}+}\| \le \frac{2C}{1-\lambda}
\left\|\mathbf{v}^{\{\infty\}}\right\|_{\B}\;,\quad
\|e_n^{\{n\}-}\| \le \frac{2C}{1-\lambda}
\left\|\mathbf{v}^{\{\infty\}}\right\|_{\B}\;,
\end{equation}
This inequality leads to the following theorem that bounds the norm of
$e_i^{\{n\}}$, the difference between the least squares solution
$v_i^{\{n\}}$ and the shadowing direction $v_i^{\{\infty\}}$.

\begin{theorem} \label{errbound}
If $n$ is sufficiently large such that $3nC\lambda^n < 1$,
then $e_i^{\{n\}}$ as defined in Equation (\ref{shadowerr}) satisfies
\[ \|e_i^{\{n\}}\| <
\frac{2C^2}{1-\lambda}\left\|\mathbf{v}^{\{\infty\}}\right\|_{\B}
(\lambda^{i} + \lambda^{n-i})\;,\quad i=1,\ldots,n\]
\end{theorem}
\begin{proof}
From the hyperbolicity assumption (\ref{hypermap}) and Lemma \ref{thm:ineq},
\[ \|e_i^{\{n\}}\| \le \|e_i^{\{n\}+}\| + \|e_i^{\{n\}-}\|
 \le C\,\lambda^{n-i}\|e_n^{\{n\}+}\| + C\,\lambda^{i}\|e_0^{\{n\}-}\| \]
The theorem is then obtained by substituting Equation (\ref{ineq})
into $\|e_n^{\{n\}+}\|$ and $\|e_0^{\{n\}-}\|$ in the inequality above.
\end{proof}

This theorem shows that $v_i^{\{n\}}$ is a good approximation of the
shadowing direction $v_i^{\{\infty\}}$ when $n$ is large and
$-\log\lambda\ll i\ll n+\log\lambda$.  The next section shows
that the approximation has a vanishing error in Equation
(\ref{lssmap}) as $n\rightarrow\infty$.  Combined with Theorem
\ref{thm:shadowmean}, we then prove a rigorous statement of Theorem
\ref{thm:lssmap}.

\section{Convergence of least squares shadowing}
\label{s:meanlss}
This section uses the results of the previous sections to prove
our main theorem.
\begin{customthm}{LSS}
For a $C^1$ map $f:\R^m\times\R\rightarrow\R^m$, assume $f(\cdot,0)$ is
bijective and defines a compact global hyperbolic attractor $\Lambda$.
For a $C^1$ map $J:\R^m\times\R\rightarrow\R$ whose infinite time
average $\langle J\rangle$ defined in Equation (\ref{mapJ}) 
is independent of the initial state $u_0\in\R^m$.
For a sequence $\{u_i,i\in\Z\}\subset\Lambda$ satisfying
$u_{i+1}=f(u_i,0)$, denote $v_i^{\{n\}}\in\R^m, i=1,\ldots,n$
as the solution to the constrained least squares problem (\ref{lss}),
Then the following limit exists and is equal to
\begin{equation} \label{formula}
 \lim_{n\rightarrow\infty}
\frac1n \sum_{i=1}^n
\left((DJ(u_i,0))\, v_i^{\{n\}} + (\partial_sJ(u_i,0))\right)
=
\frac{d\langle J\rangle}{ds}\bigg|_{s=0}\;.
\end{equation}
\end{customthm}
\begin{proof}
Because $J$ is $C^1$ and $\Lambda$ is compact,
$(DJ(u_i,0))$ is uniformly bounded, i.e., there exists $A$
such that $\|(DJ(u_i,0))\|<A$ for all $i$.
Let $e_i^{\{n\}}$ be defined as in Equation (\ref{shadowerr}), whose
norm is bounded by Theorem \ref{errbound}, then for large enough $n$,
\[\begin{aligned}
&
\left|\frac{1}n\sum_{i=1}^n \left((DJ(u_i,0))\, v_i^{\{n\}} +
(\partial_sJ(u_i,0))\right)
-
\frac{1}n\sum_{i=1}^n \left((DJ(u_i,0)) v_i^{\{\infty\}}
+ (\partial_sJ(u_i,0))\right)\right| \\
&= \left|
\frac{1}n\sum_{i=1}^n (DJ(u_i,0))\, e_i^{\{n\}} \right|
\le 
\frac{1}n\sum_{i=1}^n \left\| (DJ(u_i,0)) \right\| \|e_i^{\{n\}}\| \\
&<
\frac{1}{n}\sum_{i=1}^n 
\frac{2A\,C^2}{1-\lambda}\left\|\mathbf{v}^{\{\infty\}}\right\|_{\B}
(\lambda^{i} + \lambda^{n-i})
< \frac{1}{n}
\frac{4A\,C^2}{(1-\lambda)^2}\left\|\mathbf{v}^{\{\infty\}}\right\|_{\B}
\xrightarrow{n\rightarrow\infty}0
\end{aligned} \]
Therefore,
\[\begin{aligned}
& \lim_{n\rightarrow\infty} \frac1{n} \sum_{i=1}^n
\left((DJ(u_i,0))\, v_i^{\{n\}} + (\partial_sJ(u_i,0))\right) \\
=&
\lim_{n\rightarrow\infty} \frac1{n} \sum_{i=1}^n
\left((DJ(u_i,0))v_i^{\{\infty\}}
+ (\partial_sJ(u_i,0))\right) = \frac{d\langle J\rangle}{ds}\bigg|_{s=0}
\end{aligned} \]
by Theorem \ref{thm:shadowmean}.
\end{proof}

\section{The least squares shadowing algorithm}
\label{s:algo}
A practicable algorithm based on Theorem \ref{thm:lssmap} is the
following.
\begin{enumerate}
\item Choose large enough $n_0$ and $n$, and an arbitrary starting point
$u_{-n_0}\in\R^m$.
\item Compute $u_{i+1} = f(u_i,s), i=-n_0,\ldots,0,1,\ldots,n$.  

For large enough $n_0$, $u_1,\ldots,u_n$ are approximately on the
global attractor $\Lambda$.
\item Solve the system of linear equations
\[\left\{\begin{aligned}
& v_{i+1} = (Df(u_i,s))\, v_{i} + (\partial_sf(u_i,s)),&& i=1,\ldots,n-1\\
& w_{i-\frac12} = (Df(u_i,s))^T w_{i+\frac12} + v_i\;, && i=1,\ldots,n\\
& w_{\frac12} = w_{n+\frac12} = 0 &&\\
\end{aligned}\right.\]
which is the first order optimality condition of the constrained
least squares problem (\ref{lss}), and gives its unique solution
$v_1^{\{n\}},\ldots,v_n^{\{n\}}$.
Note that a linear relation between $w_{i-\frac12},
w_{i+\frac12}$ and $w_{i+\frac32}$ can be obtained by substituting the second
equation into the first one.  The resulting matrix system is
block-tridiagonal, where the block size is the dimension of the
dynamical system $m$.  A banded matrix solver can then be
used to solve the system.
\item Compute the desired derivative by
\begin{equation} \label{approx}
\frac{d\langle J\rangle}{ds}\approx 
\frac1n \sum_{i=1}^n
\left((DJ(u_i,0))\, v_i^{\{n\}} + (\partial_sJ(u_i,0))\right)\;.
\end{equation}
\end{enumerate}
Most of the computation time in this algorithm is spent
on solving the block-tridiagonal system in Step 3.  Because the
$n\,m\times n\,m$ matrix has a 
bandwidth of $4m-1$, the computational cost of a banded solver (e.g.,
Lapack's dgbsv routine\cite{laug}) is bounded by $O(n\,m^3)$.
Here $n$ is the length of the trajectory, and $m$ is the dimension
of the dynamical system.  $O(n\,m^3)$ is
the leading term in the number of operations of the algorithm presented
in this paper.

Theorem \ref{thm:lssmap} shows that the computed derivative is accurate
for large $n$.  The approximation error of Equation (\ref{approx})
originates from two sources,
\begin{equation}
\frac{d\langle J\rangle}{ds} - 
\frac1n \sum_{i=1}^n
\left((DJ(u_i,0))\, v_i^{\{n\}} + (\partial_sJ(u_i,0))\right) = E_1 + E_2
\end{equation}
where
\begin{equation}\label{E1} E_1 = 
\frac{d\langle J\rangle}{ds} - 
\frac1n \sum_{i=1}^n \left((DJ(u_i,0)) v_i^{\{\infty\}}
+ (\partial_sJ(u_i,0))\right)\;.
\end{equation}
Theorem \ref{thm:shadowmean} guarantees that
$E_1\xrightarrow{n\to\infty}0$.  This error represents the difference
between an ergodic mean and an average over a finite trajectory.  If the
dynamical system is mixing, the central limit theorem implies that
$E_1\sim O(n^{-\frac12})$.  The other part of the error is
\begin{equation}\label{E2} E_2 = 
\frac1n \sum_{i=1}^n\; \left(DJ(u_i,0)\right)
\left(v_i^{\{\infty\}} - v_i^{\{n\}}\right)
\;.
\end{equation}
Theorem \ref{errbound} guarantees that $E_2\sim O(n^{-1})$.  Because
$E_1$ has a slower rate of decay, the rate of convergence of the
algorithm presented in this paper is $O(n^{-\frac12})$ for
sufficiently large $n$.

\section{A numerical demonstration}
The algorithm described in Section \ref{s:algo}
is implemented in the Python code {\bf lssmap}, available at
\underline{https://github.com/qiqi/lssmap}.
\footnote{
All the numerical results in this section is obtained
by running revision
fa82e4241ad3d2a62603224d4afa54c9500f6224 of this code hosted on github}.

\begin{figure}[htb!]\centering
\includegraphics[width=2.5in]{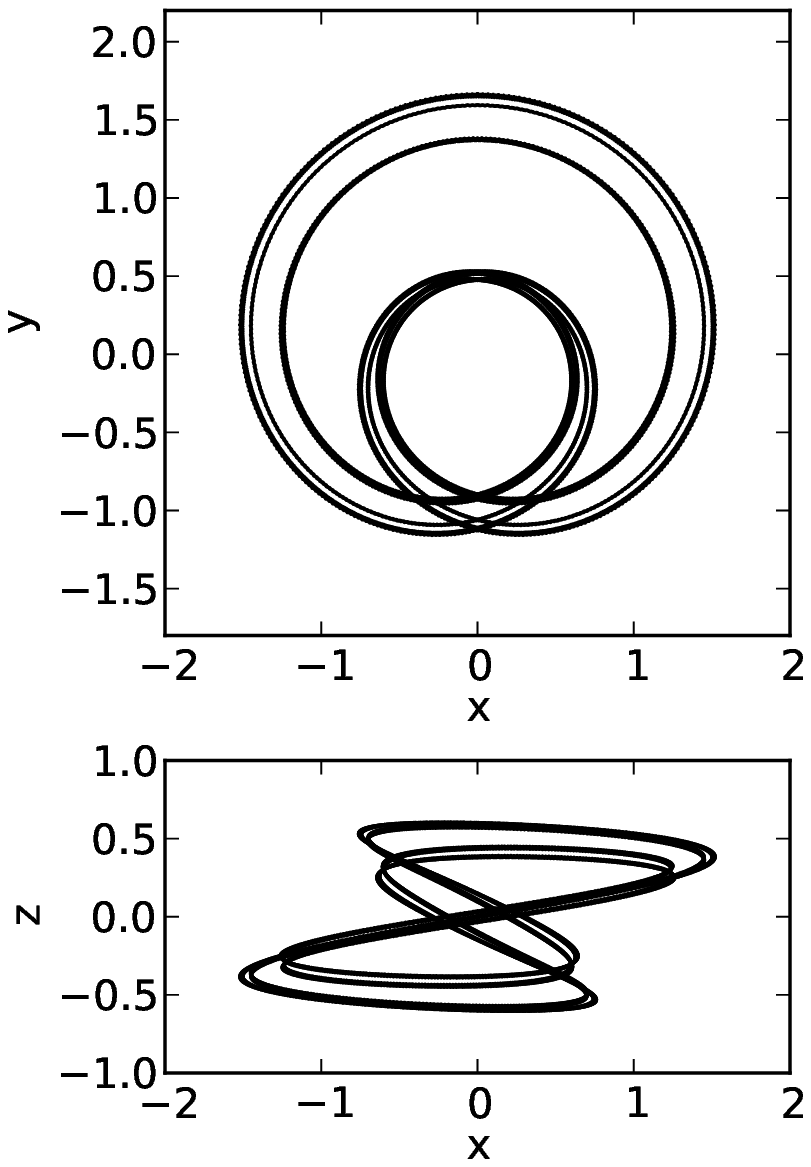}
\includegraphics[width=2.5in]{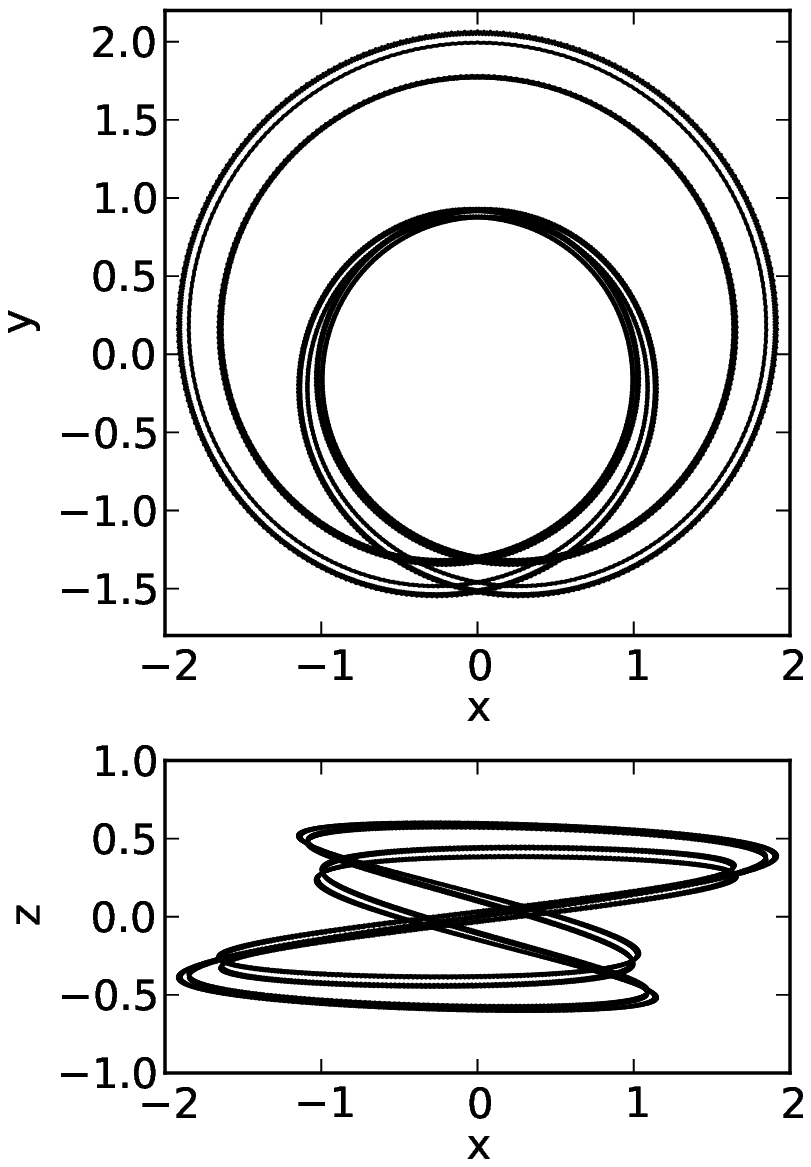}
\caption{Visualization of the Smale-Williams solenoid attractor defined
by the map in Equation (\ref{solenoid}).
The left plots show the attractor at $s=1$.
The right plots show the attractor at $s=1.4$.}
\label{f:attractor}
\end{figure}

\begin{figure}[htb!]\centering
\includegraphics[width=4in]{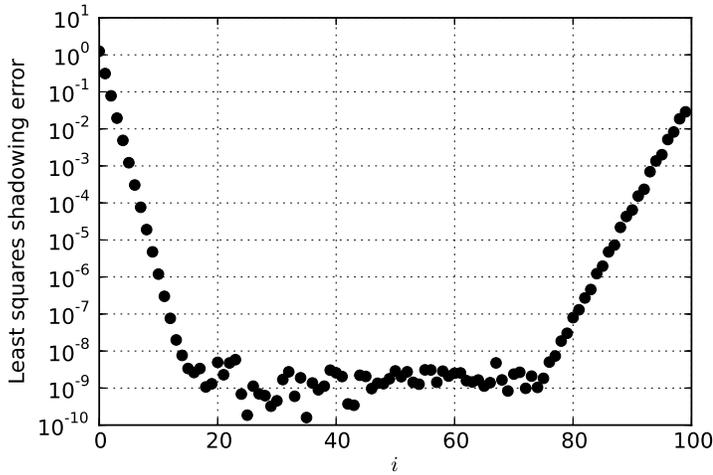}
\caption{The $l^2$ norm of the least squares shadowing error
$e_i^{\{n\}} = v_i^{\{n\}} - v_i^{\{\infty\}}$ for a trajectory of
length $n=100$ at $s=2.0$.}
\label{f:err}
\end{figure}

\begin{figure}[htb!]\centering
\includegraphics[width=4.5in]{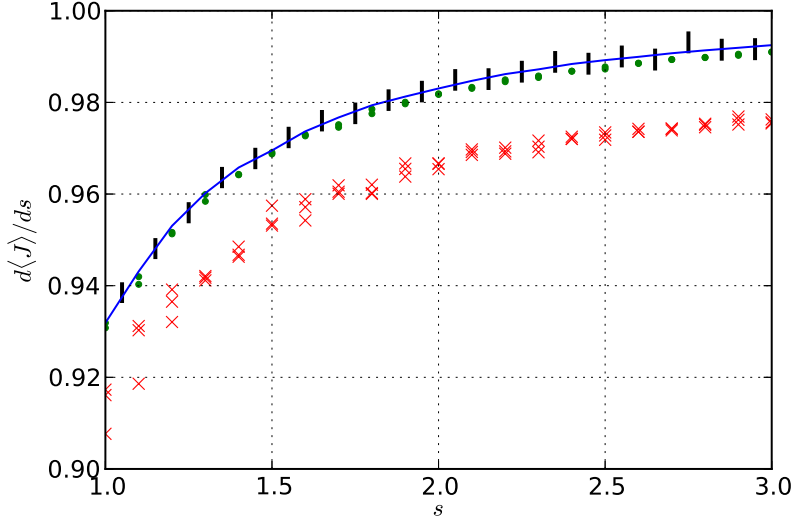}
\caption{
$d\langle J\rangle/ds$ computed using the least squares shadowing algorithm.
Red X's represent those computed with trajectories of length $n=100$.
Green dots represent those computed with $n=1000$.  Blue lines represent
those computed with $n=10000$.  Each calculation is repeated several
times at the same value of $s$.  The black bars represent the $3\sigma$
confidence interval of finite difference derivatives.  Each finite
difference derivative is computed by differencing the mean of
10000 trajectories at $s+0.05$ and the mean of 10000
trajectories at $s-0.05$.  Each of these 20000 trajectories has length 10000.}
\label{f:lss}
\end{figure}

\begin{figure}[htb!]\centering
\includegraphics[width=3.5in]{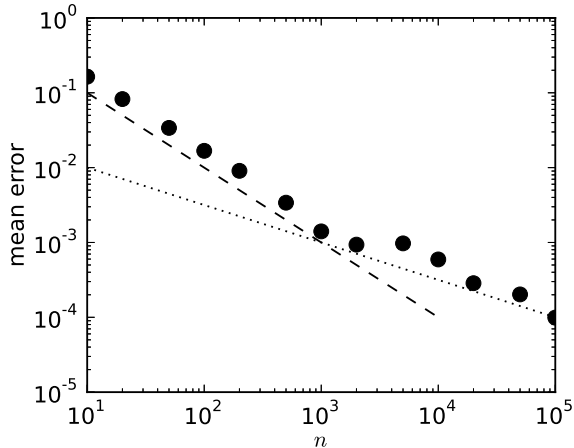}
\caption{The error in the computed derivative
$d\langle J\rangle/ds$ at $s=1$ as a function of the trajectory length $n$.
Each dot is an averaged value of absolute value of the error computed
with 16 random trajectories of the same length.  The dashed line
indicate $O(n^{-1})$ rate of decay; the dotted line indicates
$O(n^{-\frac12})$ rate of decay.}
\label{f:converge}
\end{figure}

The algorithm is tested on the Smale-Williams solenoid attractor.
The map that defines this attractor in cylindrical coordinates is
\footnote{
Although the map is defined on cylindrical coordinates, the $L^2$ norm
in $R^3$ is the Euclidean distance in Cartesian coordinates.
In the numerical implementation of this map,
the Cartesian coordinates of $u_n$ is transformed to cylindrical coordinates,
then the map is applied to obtain $u_{n+1}$ in
cylindrical coordinates before it is transformed back to Cartesian
coordinates.}
\begin{equation} \label{solenoid}
   u_{n+1}
 = \begin{bmatrix}r_{n+1}\\ \theta_{n+1}\\z_{n+1}\end{bmatrix}
 = \begin{bmatrix}
 s + (r_{n} - s)/4 + (\cos\theta_n)/2\\
 2 \theta_n\\
 z_n/4 + (\sin\theta_n)/2 \end{bmatrix}\;.
\end{equation}
The map has a single parameter $s$, whose effect is qualitatively shown
in Figure \ref{f:attractor}.
We define the quantity of interest
\[ J(u) = \sqrt{r^2 + z^2}\;, \]
and focus on computing the
derivative of the long time averaged quantity of interest $\langle
J\rangle$ to the parameter $s$.

This particular map is chosen such that the shadowing direction has
a rare analytic form.  It is straightforward to verify that the constant
sequence $v_i^{\{\infty\}} \equiv [r=1, \theta=0, z=0]$
satisfies the tangent map
$v_{i+1}^{\{\infty\}} = (Df(u_i,0)) v_i^{\{\infty\}} +
(\partial_sf(u_i,0)) $ for any sequence $\{u_i\}$.  
This analytic form of the shadowing direction allows us to numerically
evaluate the least squares shadowing error $e_i^{\{n\}}$ as defined in
Equation (\ref{shadowerr}).  Figure \ref{f:err} shows that the error
is order 1 at both the beginning and end of a trajectory, but decreases
exponentially to numerical precision towards the middle portion of the
trajectory.  This trend is consistent with the error bound provided
by Theorem \ref{errbound}.

The values of $d\langle J\rangle/ds$ computed from the least squares
shadowing algorithm is plotted in Figure \ref{f:lss} and compared
against finite difference derivatives.  The derivatives computed on
trajectories of length $n=100$ has significant error because
$e_i^{\{n\}}$ is large on a significant portion of the trajectory.
The derivatives computed with $n=1000$ and $10000$ appear to be at least
as accurate as the finite difference values.  It is worth noting that
each finite difference calculation involves trajectories of total length
$200,000,000$, and takes orders of magnitude longer computation time
than a least squares shadowing calculation with $n=10000$.

Figure \ref{f:converge} illustrates that the least squares
shadowing algorithm converges at a rate of $O(n^{-1})$ at relatively
small values of $n$, then transitions into a rate of $O(n^{-\frac12})$
at higher $n$.
\footnote{
The truth value of $d\langle J\rangle/ds\,|_{s=1}$
used in this convergence analysis is
$0.931450\pm 0.000017$ with $99.7\%$ confidence.  This value is obtained
by averaging over 1100 least squares shadowing calculations, each of
length 100000.  The first 20 steps and the last 20 steps of each
trajectory is removed from the averaging in order to remove the bias caused by
$E_1$.  The value of 20 is motivated by Figure \ref{f:err}.
}
This behavior is consistent with the error analysis
in Section \ref{s:algo}.  For small $n$, $E_1$ as in Equation (\ref{E1}),
which has a decay rate of $O(n^{-1})$, dominates.  For larger $n$,
$E_2$ as in Equation (\ref{E2}), which has a slower decay rate of
$O(n^{-\frac12})$, dominates.  They lead to a two-stage convergence
pattern as seen in Figure \ref{f:converge}.

\section*{Acknowledgment}
The author thanks financial support from AFOSR support under STTR contract
FA9550-12-C-0065 through Dr. Fariba Farhoo, and NASA funding
through technical monitor Dr. Harold Atkins.
The author gratefully acknowledges David Moro and Dr. Si Li for helpful
discussion on the proofs.

\bibliography{master}

\end{document}